\documentclass[a4paper,12pt]{amsart}
\usepackage{amsmath,amsthm,amssymb}
\usepackage{amsfonts}
\usepackage{txfonts}
\usepackage{latexsym}
\usepackage[dvipdfmx]{graphics,color}
\usepackage{eucal}
\usepackage{mathrsfs}
\usepackage{url}
\theoremstyle{plain}
\newtheorem{theorem}{Theorem}[section]
\newtheorem{proposition}[theorem]{Proposition}
\newtheorem{lemma}[theorem]{Lemma}
\newtheorem{corollary}[theorem]{Corollary}

\theoremstyle{definition}
\newtheorem{definition}[theorem]{Definition}

\makeatletter
\renewenvironment{proof}[1][\proofname]{\par
  \normalfont
  \topsep6\p@\@plus6\p@ \trivlist
  \item[\hskip\labelsep{\bfseries #1}\@addpunct{\bfseries.}]\ignorespaces
}{%
  \endtrivlist
}
\renewcommand{\proofname}{proof}
\theoremstyle{remark}
\newtheorem{remark}[theorem]{Remark}

\numberwithin{equation}{section}

\title{Non-unital algebraic $K$-theory and almost mathematics} 
\date{\today} 
\author{Yuki Kato}
\address{National institute of technology, Ube college, 
	      2-14-1, Tokiwadai, Ube, Yamaguchi, JAPAN 755-8555.}
\email{ykato@ube-k.ac.jp}

\newcommand{\Z}{\mathbb{Z}}

\newcommand{\Mod}{\mathrm{Mod}}
\newcommand{\CAlg}{\mathrm{CAlg}}

\newcommand{\Fun}{\mathrm{Fun}}
\newcommand{\PMod}{\mathrm{PMod}}

\newcommand{\tm}{\tilde{\mathfrak{m}}}

\usepackage{enumerate}
\usepackage{array}
\usepackage[all]{xy} 
\usepackage{delarray}

\begin{document}
\begin{abstract} 
 The Gersten conjecture is still an open problem of algebraic $K$-theory
 for mixed characteristic discrete valuation rings. In this paper, we
 establish non-unital algebraic $K$-theory which is modified to become
 an exact functor from the category of non-unital algebras to the stable
 $\infty$-category of spectra.  We  prove that for any almost unital algebra, the non-unital $K$-theory  homotopically decomposes into the non-unital $K$-theory the  corresponding ideal and the residue algebra, implying the Gersten
 property of non-unital $K$-theory of the the corresponding ideal.
\end{abstract}
\maketitle

\section{Introduction}
\label{intro}
Let $V$ be a Noetherian regular local ring and $F$ denote the fractional field. 
Then the inclusion $j:V \to F$ induces the pullback 
\[
 j^* :K_n(V) \to K_n(F)
\] 
of algebraic $K$-groups for any integer $n$. Gersten~\cite{Gersten}
 conjectured the induced homomorphism $j^*$ is injective for any integer
 $n$. This conjecture is called the {\it Gersten conjecture} for
 algebraic $K$-theory. In the case $V$ of positive characteristic,
 Panin~\cite{Panin} proved Gersten's conjecture is true by using
 Quillen's result~\cite{Quillen} of the case $V$ is essentially smooth
 over a field and Popescu's result~\cite{Popescu}.  In the mixed
 characteristic case, the Gersten conjecture is still open.

This paper considers the Gersten conjecture for valuation rings $V$ with
 idempotent maximal ideals and non-unital algebras for non-unital
 algebraic $K$-theory $K^{\rm nu}$ which is modified to become an exact
 functor from the category of non-unital algebras to the stable
 $\infty$-category of spectra. Let $\mathfrak{m}$ be an idempotent
 maximal ideal of $V$ and $A$ an almost $V$-algebra.  We assume that
 $\mathfrak{m}$ is a flat $V$-module. We prove that the non-unital
 $K$-theory $K^{\rm nu}(A)$ homotopically decompose into the non-unital
 $K$-theory of the ideal $\mathfrak{m}A$ and of the residue algebra $A/\mathfrak{m}A$ (Theorem~\ref{mainAlmost}): $ K^{\rm nu}(A)\simeq K^{\rm nu}(\mathfrak{m}A) \oplus  K^{\rm nu}(A/\mathfrak{m}A)$,  
implying that one has
 a decomposition
\[
   K^{\rm nu}(A \otimes_V F) \simeq K^{\rm nu}(\mathfrak{m}A)
 \oplus K^{\rm nu}((A \otimes_{V}F) /\mathfrak{m}A ),
\] 
where $K^{\rm nu}((A \otimes_{V}F) /\mathfrak{m}A ) $ denotes the
homotopy cofiber of the map $ K^{\rm nu}(\mathfrak{m}A) \to K^{\rm nu}(A
\otimes_{V}F)$ of spectra. In particular, the induced map $ K^{\rm
nu}(\mathfrak{m}) \to K^{\rm nu}(F) $ is homotopically split injective.
%
%
%

\section{Non-unital algebraic $K$-theory}
\subsection{A review of non-unital rings}
Let $A$ be a commutative ring. Then the category of $A$-modules is
equivalent to the category of $\Z \oplus A$-modules.
\begin{proposition}[\cite{Q}]
If $A$ has a multiplicative unit $1_A$, then one has a categorical equivalence 
$\Mod_{\Z \oplus A} \simeq \Mod_\Z \times \Mod_A $ which induces an equivalence $\PMod_{\Z \oplus A} \simeq \PMod_\Z \times \PMod_A$.   
\end{proposition}
\begin{proof}
The distinguish element $e=(0,\,1_A ) \in \Z \oplus A $ is clearly a projector. For any $\Z \oplus A$-module $M$, $M$ decomposes into $M \simeq eM \oplus (1-e)M$. If $M$ is a projective $\Z \oplus A$-module, $eM $ and $(1-e)M$ are also projective $\Z \oplus A$-modules. The module $eM$ is killed by $\Z$, having an $A$-module structure. Hence $eM$ is also a projective $A$-module. qed
\end{proof}

\begin{definition}
Let $A$ be a commutative ring. The $K(A)$ is defined to be the homotopy fiber of the map $K(\Z \oplus A) \to K(\Z )$ induced by the projection. 
\end{definition}

If $A$ has a multiplicative unit $1_A$, then $K(A)$ is an ordinal $K$-theory. Indeed, the functor $A \otimes_{\Z \oplus A}-: \Mod_{\Z \oplus A} \to \Mod_A$
induces a projector $K(\Z \oplus A) \to K(A)$ of spectra.

\begin{theorem}[\cite{weibel2} Theorem 2.1 (Excision for ideals)]
\label{W2}
Let $I$ be an ideal of a ring $A$. Then 
\[
  KH(I) \to KH(A) \to KH(A/I)
\]
 is a fiber sequence. \qed
\end{theorem}
In order to use the above property of homotopy invariant $K$-theory for
any exact sequence of non-unital algebra, we adjust the $K$-theory
spectrum as the following: Let $\mathrm{Sp}$ denote the stable
$\infty$-category of spectra and $\Fun^{\rm rex}( \CAlg^{\rm nu},\,
\mathrm{Sp})$ the full subcategory of $\Fun( \CAlg^{\rm nu},\,
\mathrm{Sp})$ spanned by right exact functors. Then the inclusion
$\Fun^{\rm rex}( \CAlg^{\rm nu},\, \mathrm{Sp}) \to \Fun( \CAlg^{\rm
nu},\, \mathrm{Sp})$ admits a left adjoint
\[
P_1:  \Fun( \CAlg^{\rm nu},\, \mathrm{Sp}) \to \Fun^{\rm rex}(
\CAlg^{\rm nu},\, \mathrm{Sp})
\]
by~\cite[p.759, Theorem 6.1.1.10]{HA}.


\begin{lemma}
\label{split-lemma}
Let $\mathcal{C}$ be a stable $\infty$-category. Given two fiber sequences $X \to Y \to Z$ and $X \to Y' \to Z'$ of objects of $\mathcal{C}$ and a homotopically commutative square:
\[
 \xymatrix@1{ Y  \ar[r] \ar[d] & Z \ar[d] \\
Y'\ar[r] & Z',
}
\]
the square is both homotopy Cartesian and coCartesian, and the following properties are equivalent:
\begin{itemize}
 \item[(1)] The object $Y$ homotopically decomposes into the direct sum
	    of $X$ and $Z$.
 \item[(2)] The object $Y'$ homotopically decomposes into the direct sum
	    of $X$ and $Z'$.
\end{itemize}
\end{lemma}
\begin{proof}
 Note that, in a stable $\infty$-category, a homotopy coCartesian product is also a homotopy Cartesian product. By a formal argument, one has $Z' \simeq Y' \amalg_{Y }  Z$, entailing $Y \simeq Y' \times_{Z'} Z$ : Indeed, one has a chain of weak equivalences: $ Y' \amalg_{Y }  Z \simeq Y' \amalg_{Y}  ( Y \amalg_X 0  ) \simeq    Y' \amalg_{X} 0 \simeq Z'$.

If condition (1) holds, $Y$ is weakly equivalent to the coproduct of $X$ and $Z$. Hence, we have a chain of weak equivalences: $
 X \amalg Z' \simeq X \amalg ( Z \amalg_{Y} Y' ) \simeq   (X \amalg  Z) \amalg_{Y} Y' \simeq Y \amalg_{Y}Y' \simeq Y'.$ 

Conversely, condition (2) implies $Y' \simeq X \times Z'$, giving a chain of equivalences: $ Y \simeq Y' \times_{Z'} Z \simeq (X \times Z'  )  \times_{Z'} Z \simeq X \times Z$. \qed  
\end{proof}

By Lemma~\ref{split-lemma}, we obtain the following proposition:
\begin{proposition}
 \label{split}
Let $\mathfrak{m} \subset V \subset F$ be a sequence of commutative
rings. For any $V$-algebra $A$, write $A_F = A \otimes_V F$. Let $\mathcal{F}:\CAlg^{\rm nu} \to  \mathrm{Sp} $ be a functor. 
Then the spectrum $P_1(\mathcal{F})(A)$ decomposes into a direct sum of
$P_1(\mathcal{F})(\mathfrak{m}A)$ and $P_1(\mathcal{F})(A /\mathfrak{m}A)$ if and only if  $P_1(\mathcal{F})(A_F)$ has also a decomposition into a direct sum of $P_1(\mathcal{F})(\mathfrak{m}A)$ and $P_1(\mathcal{F})(A_F/ \mathfrak{m}A)$, where $P_1(\mathcal{F})(A_F/ \mathfrak{m}A)$ denotes the homotopy cofiber of $P_1(\mathcal{F})(\mathfrak{m}A) \to P_1(\mathcal{F})(A_F) $. \qed
\end{proposition}

%
%
%

\begin{corollary}
\label{Gersten}
Let $\mathfrak{m} \subset V \subset F$ be a sequence of commutative
rings. Assume that $P_1(\mathcal{F})(A)$ decomposes into a direct sum of
$P_1(\mathcal{F})(\mathfrak{m}A)$ and $P_1(\mathcal{F})(A/\mathfrak{m}A)$.  Then the
pullback $j^*:P_1(\mathcal{F})(A) \to P_1(\mathcal{F})(A_F)$ induces an injection $\pi_n(P_1(\mathcal{F})(A)) \to \pi_n(P_1(\mathcal{F})(A_F))$ for each $n \ge 0$ if and only if the induced map
$\pi_n(P_1(\mathcal{F})(A/ \mathfrak{m} A)) \to \pi_n(P_1(\mathcal{F})(A_F/\mathfrak{m}A))$ is
injective for each $n \ge 0$. \qed
\end{corollary}

\section{Non-unital $K$-theory of almost mathematics}
\subsection{The almost Gersten property of non-unital $K$-theory}

Let $V$ be a commutative unital ring with an idempotent ideal $\mathfrak{m}$. A $V$-module $M$ is said to be {\it almost zero} if $\mathfrak{m}M=0$. 
A $V$-homomorphism $f:M \to N$ of $V$-modules is called {\it an almost isomorphism} if both the kernel and the cokernel of $f$ are almost zero. An {\it almost $V$-module} is an object of the (bi)localization of the category of $V$-modules by the Serre subcategory spanned by almost zero modules, and an {\it almost $V$-algebra} is a commutative algebra object of the category of almost $V$-modules. Write $\tm = \mathfrak{m} \otimes_V \mathfrak{m}$. In this section, we always assume that $\mathfrak{m}$ is a flat $V$-module. Then the multiplication $\tm \to \mathfrak{m}$ is an isomorphism.

We define (recall) of the definition of almost
$K$-theory~\cite{K-almost}. Let $A$ be an almost $V$-algebra and
$\mathrm{APerf}(A)$ denote the full subtriangulated category generated
by $\tm \otimes_V A$ of the derived category of $A$-modules. Further,
let $\mathrm{Perf}^+(A)$ denote the full subtriangulated category
generated by $A$ and $\tm \otimes_V A$ of the derived category of
$A$-modules.

\begin{definition}
Let $A$ be an almost $V$-algebra and $K^{\rm nu}(A)^{\rm al}$ denote the
non-unital $K$-theory spectrum of the triangulated category of
$\mathrm{APerf}(A)$. We call $K^{\rm nu}(A)^{\rm al}$ the {\it
non-unital almost $K$-theory spectrum} of $A$. Further, let $K^{\rm
nu}(A)^{+}$ denote the non-unital $K$-theory spectrum of the
triangulated category $\mathrm{Perf}^+(A)$.
\end{definition}

Consider a diagram $B \leftarrow A \rightarrow C$ of non-unital
$V$-algebras. Let $K$ denote the kernel the argumentation $(V \oplus B)
\otimes_{V \oplus A} (V \oplus C) \to V$. Then the unitalization $V
\oplus K$ is isomorphic to $ (V \oplus B) \otimes_{V \oplus A} (V \oplus
A)$ by the equivalence $ V \oplus (-)$ from the category of non-unital
$V$-algebras to the category of unital augmented
$V$-algebras. Therefore $K$ represents the colimit of the diagram $B
\leftarrow A \rightarrow C$. We write $B \Box_{A} C= K$ for the colimit
of $B \leftarrow A \rightarrow C$. By definition of non-unital $K$-theory, 
the induced square
\[
   \xymatrix@1{ K^{\rm nu}(A) \ar[r] \ar[d]&   K^{\rm nu}(B) \ar[d] \\  
  K^{\rm nu}(C)\ar[r] &   K^{\rm nu}( B \Box_{A} C )
}
\] 
is homotopy coCartesian. 
 
\begin{lemma}
\label{cofiber} Let $A$ be an almost $V$-algebra and $K^{\rm nu}(A
/\mathfrak{m} A) $ (resp. $K^{\rm nu}(A / \mathfrak{m}A)^+$) denote the
homotopy cofiber of $ K^{\rm nu}(\mathfrak{m}A) \to K^{\rm nu}(A) $
(resp. $ K^{\rm nu}(\mathfrak{m}A)^{+} \to K^{\rm nu}(A)^{+}$). Then the
augmentation $V \otimes_{V \oplus \mathfrak{m}A} (V \oplus A) \to V$ is
an almost isomorphism. Furthermore, the induced map $K^{\rm nu}(A
/\mathfrak{m} A) \to K^{\rm nu}(A /\mathfrak{m} A)^+ $ is a weak
equivalence of spectra.
\end{lemma}
\begin{proof}
The proof is a formal argument: One has a chain of isomorphisms 
\begin{multline*}
 \tm \otimes_{V} ( V \otimes_{V \oplus \mathfrak{m}A} (V \oplus A)   ) 
\simeq    \tm \otimes_{V \oplus \mathfrak{m}A} (V \oplus A) \simeq  (\tm \otimes_{V \oplus \mathfrak{m}A} V) \oplus  ( \tm \otimes_{V \oplus \mathfrak{m}A}  A) \\
\simeq      (\tm \otimes_{V \oplus \mathfrak{m}A} V  )\oplus (  \tm \otimes_{V \oplus \mathfrak{m}A} \mathfrak{m} A) \simeq   \tm \otimes_{V \oplus \mathfrak{m}A} (V \oplus \mathfrak{m}A) \simeq \tm. 
\end{multline*}

The kernel $A \Box_{\mathfrak{m}A} 0$ of the  augmentation $V \otimes_{V \oplus \mathfrak{m}A} (V \oplus A) \to V$ represents the homotopy cofiber $K^{\rm nu}(A / \mathfrak{m}A) $ (resp. $K^{\rm nu}(A / \mathfrak{m}A)^+ $). Then the projection $A \Box_{\mathfrak{m}A} 0 \to (A \Box_{\mathfrak{m}A} 0) \otimes_V V / \mathfrak{m}$ is an isomorphism. Therefore $K(A \Box_{\mathfrak{m}A} 0) \to K( (V/ \mathfrak{m})    \otimes_{V \oplus \mathfrak{m}A} (V \oplus A)  ) \to K(V/ \mathfrak{m}) $ is a homotopy fiber sequence. Since the inclusion functors $\mathrm{Perf}( (V/ \mathfrak{m})    \otimes_{V \oplus \mathfrak{m}A} (V \oplus A)  )  \to\mathrm{Perf}^+( (V/ \mathfrak{m})    \otimes_{V \oplus \mathfrak{m}A} (V \oplus A)  ) $ and  $\mathrm{Perf}( V/ \mathfrak{m}  )  \to\mathrm{Perf}^+( V/ \mathfrak{m})$ are categorical equivalence, the induced morphism $K^{ \rm nu}(A \Box_{\mathfrak{m}A} 0) \to K^{ \rm nu}(A \Box_{\mathfrak{m}A} 0)^+$ is a weak equivalence. \qed
\end{proof}

\begin{theorem}[c.f.\cite{K-almost} Theorem 3.15 and Theorem 3.21]
\label{K-almostMainA} Let $A$ be an almost $V$-algebra. Then the non-unital $K$-theory $K^{\rm nu}(A)^{+}$ homotopically decomposes into the product of $K^{\rm nu}(A)^{\rm al}$ and $K^{\rm nu}(A/\mathfrak{m}A)$.   
\end{theorem}
\begin{proof}
The proof is a similar argument of the proof of \cite[Theorem
3.11]{K-almost}. The functor $\tm \otimes_V (-):\mathrm{Perf}^+(A)\to
\mathrm{Perf}^+(A) $ is categorical idempotent, and the essential image
is equivalent to $\mathrm{APerf}(A)$. Therefore the functor $\tm
\otimes_V (-):\mathrm{Perf}^+(A)\to \mathrm{Perf}^+(A) $ induces a
homotopically splitting $K( \tm \otimes_V (-)) :K^{\rm nu}(A)^{\rm al}
\to K^{\rm nu}(A)^{+}$ and a decomposition $ K^{\rm nu}(A)^{+} \simeq
K^{\rm nu}(A)^{\rm al} \oplus K^{\rm nu}(A)^{ \mathfrak{m}}$, where
$K^{\rm nu}(A)^{\mathfrak{m}}$ denotes the homotopy fiber of $ K^{\rm
nu}(A)^{+} \to K^{\rm nu}(A)^{\rm al}$.

Let $K^{\rm nu}(A/\mathfrak{m}A)^+$ (resp. $K^{\rm nu}( A /\mathfrak{m}A)^{\rm al} $) denote the homotopy cofiber of $ K^{\rm nu}(\mathfrak{m}A)^+ \to   K^{\rm nu}(A)^+$ (resp. $ K^{\rm nu}(\mathfrak{m}A)^{\rm al} \to   K^{\rm nu}(A)^{\rm al}$). By Lemma~\ref{cofiber},  $K^{\rm nu}( A /\mathfrak{m}A)^{\rm al}$ is contractible, entailing that  $ K^{\rm nu}(\mathfrak{m}A)^{\rm al} \to K^{\rm nu}(A)^{\rm al} $ is a weak equivalence.    


Next, we show that the induced map $ K^{\rm nu}(\mathfrak{m}A)^{+}  \to    K^{\rm nu}(\mathfrak{m}A)^{\rm al}$ is a weak equivalence. Consider the unitalization $V \oplus \mathfrak{m}A$ and the argumentation $\varepsilon: V \oplus \mathfrak{m}A \to V$. Then $ K^{\rm nu}( V \oplus \mathfrak{m}A)^{+} $ is decompose into the direct sum $ K^{\rm nu}(\mathfrak{m}A)^{+} \oplus K^{\rm nu}(V)^{+}$.   Note that $E \simeq E \ \otimes_{V \oplus \mathfrak{m}A} (V \oplus \mathfrak{m}A  ) \simeq (E\otimes_{V \oplus \mathfrak{m}A} V) \oplus (E  \otimes_{V \oplus \mathfrak{m}A}  \mathfrak{m}A  )  $ for any $V \oplus \mathfrak{m}A $-complex $E$. Therefore those projections $K^{\rm nu}( V \oplus \mathfrak{m}A)^{+}  \to  K^{\rm nu}(V)^{+} $ and $K^{\rm nu}( V \oplus \mathfrak{m}A)^{+}  \to  K^{\rm nu}(\mathfrak{m}A)^{+} $ is induced by those functors $ (-) \otimes_{V \oplus \mathfrak{m}A } V$ and $ (-) \otimes_{V \oplus \mathfrak{m}A } \mathfrak{m}A$, respectively. Furthermore, for any $V \oplus \mathfrak{m}A $-complex $E$, the canonical morphism $\tm \otimes_V     E \otimes_{V \oplus \mathfrak{m}A } \mathfrak{m}A  \to  E \otimes_{V \oplus \mathfrak{m}A } \mathfrak{m}A $ is already an isomorphism by $\tm \otimes_V \mathfrak{m}A \simeq \mathfrak{m}^3 A=  \mathfrak{m}A$. Hence, we have a weak equivalence $K^{\rm nu}(\mathfrak{m}A)^{+}  \simeq   K^{\rm nu}(\mathfrak{m}A)^{\rm al}$.  

Finally, one has weak equivalences: $ K^{\rm nu}(\mathfrak{m}A)^{+}  \simeq  K^{\rm nu}(\mathfrak{m}A)^{\rm al} \simeq  K^{\rm nu}(A)^{\rm al}  $ and $K^{\rm nu}(A)^{+} \simeq K^{\rm nu}(A)^{\rm al} \oplus K^{\rm nu}(A)^{\mathfrak{m}} \simeq K^{\rm nu}(\mathfrak{m}A)^{+} \oplus K^{\rm nu}(A)^{\mathfrak{m}} \simeq K^{\rm nu}(\mathfrak{m}A)^{+} \oplus 
K^{\rm nu}(A/\mathfrak{m}A)^+$, giving us the conclusion by the second part of Lemma~\ref{cofiber}. 
\qed 
\end{proof}

\begin{theorem}
\label{mainAlmost}
Let $A$ be  an almost $V$-algebra.  Then the non-unital $K$-theory $K^{\rm nu}(A)$ is weakly equivalent to the homotopy product of $K^{\rm nu}(\mathfrak{m}A)$ and $K^{\rm nu}(A/\mathfrak{m}A)$.  
\end{theorem}
\begin{proof}
By the argument of the proof of Theorem~\ref{K-almostMainA}, all of the canonical maps $K^{\rm nu}(\mathfrak{m}A)^{\rm al} \to K^{\rm nu}(\mathfrak{m}A)^+ \to  K^{\rm nu}(A)^{\rm al}$ are weak equivalences,  implying that one has a homotopy Cartesian
square
\[
\xymatrix@1{  K^{\rm nu}(\mathfrak{m}A) \ar[r] \ar[d]& K^{\rm nu}(A) \ar[d] \\
K^{\rm nu}(A)^{\rm al} \ar[r]& K^{\rm nu}(A)^+,
} 
\]
where both of the cofibers of horizontal maps are the same $K^{\rm nu}(A / \mathfrak{m}A)$ up to weak equivalence. Since the lower horizontal map is homotopically split, the upper one is also homotopically split by Lemma~\ref{split-lemma}. \qed     
\end{proof}

\begin{corollary}
\label{nun-unitalGersten}
Let $V$ be a valuation ring with an idempotent maximal ideal $\mathfrak{m}$ and $F$ denote the fractional field of $V$.  Assume that $\mathfrak{m}$ is
flat. Let $A$ be an almost $V$-algebra. Then the canonical morphism $K^{\rm nu}(\mathfrak{m}A) \to K^{\rm nu}(A \otimes_V F)$ is homotopically split injective. Furthermore, for any $F$-algebra $B$, we have an canonical weak equivalence
\[
  K^{\rm nu}(B) \simeq K^{\rm nu}(\mathfrak{m}A) \oplus K^{\rm nu}(B/\mathfrak{m}A), 
\]
  where $  K^{\rm nu} (B/\mathfrak{m}A) $ denotes the homotopy cofiber of $ K^{\rm nu}(\mathfrak{m}A) \to  K^{\rm nu}(B)$.
\end{corollary}
\begin{proof}
This corollary is immediately obtained by Theorem~\ref{mainAlmost}: Since all of those functors $\mathrm{Perf}(B) \overset{\tm \otimes_V (-)}{\to}     \mathrm{APerf}(B) \to \mathrm{Perf}^+(B)$ are canonically categorical equivalences by the isomorphism: $\mathfrak{m} \otimes_V B \simeq B$, one has weak equivalences $:K^{\rm nu}(B) \simeq K^{\rm nu}(B)^{\rm al} \simeq K^{\rm nu}(B)^{+}$. Therefore one has a homotopy coCartesian square 
\[
\xymatrix@1{ 
 K^{\rm nu}(A) \ar[r] \ar[d] & K^{\rm nu}(A/\mathfrak{m}A) \ar[d] \\
 K^{\rm nu}(B) \ar[r] & K^{\rm nu} (B/\mathfrak{m}A).
}
\] 
 \qed
\end{proof}
In particular, one has the following splittings: $K^{\rm nu}(V) \simeq
K^{\rm nu}(\mathfrak{m}) \oplus K^{\rm nu}(V /\mathfrak{m})$ and $K^{\rm
nu}(F) \simeq K^{\rm nu}(\mathfrak{m}) \oplus K^{\rm nu}(F
/\mathfrak{m})$. 
\begin{corollary}
\label{modM}
Let $V$ be a valuation ring with an idempotent maximal ideal $\mathfrak{m}$ and $F$ denote the fractional field of $V$. Then the pullback 
$K^{\rm nu}(V) \to K^{\rm nu}(F)$ induces injections between their all homotopy groups if and only if $K^{\rm nu}(V/\mathfrak{m}) \to K^{\rm nu}(F/\mathfrak{m})$ has the same property. \qed
\end{corollary}

\subsection{A remark on the case an integral  perfectoid valuation ring}
We will apply Corollary~\ref{modM} to the case $V$ a perfectoid valuation ring. Recall the definition of perfectoid algebra: 
\begin{definition}
Let $F$ be a complete non-Archimedian non-discrete valuation field of
rank $1$, and $V$ denote the subring of powerbounded elements.  We
say that $F$ is a {\it perfectoid field} if the Frobenius $\Phi:V/ p V \to V /p V$ is surjective, where $p>0$ the characteristic of the residue field of $V$.
\end{definition} 

In this case, it is known that the maximal ideal $\mathfrak{m}$ of $V$
is flat and idempotent (See~\cite[Example 4.1.3.]{Bhatt-lecture}). For
any $V$-algebra $A$,let $A^\flat$ denote the tilting algebra
$\varprojlim_{ x \mapsto x^p} A/ p A$ of $A$. The tilting ideal
$\mathfrak{m}^\flat \subset V^\flat $ is a flat $V^\flat$-module as
$\mathfrak{m}$ is. Note that $A^\flat / \mathfrak{m}^\flat A^\flat \to A
/ \mathfrak{m} A $ is an isomorphism of commutative unital rings of
positive characteristic.

Under the assumption that we are given an weak equivalence: $K^{\rm nu}(F^\flat /V^\flat ) \simeq  K^{\rm nu}(F/V)$, which is weaker than $ K^{\rm nu}(F/V) \simeq K^{\rm nu}(k)[1] \simeq K^{\rm nu}(F^\flat /V^\flat )$, one has the following:  
\begin{proposition}
\label{local-Almost} Let $V$ be a mixed characteristic integral
perfectoid valuation ring with an idempotent maximal ideal
$\mathfrak{m}$ and $F$ denote the fractional field of $V$. Assume that the non-unital $K$-theories $K^{\rm nu}(V)$
and $K^{\rm nu}(V^\flat)$ hold the condition: We are given
weak equivalences $K^{\rm nu}(F^\flat /V^\flat ) \simeq  K^{\rm nu}(F/V)$. Then the pullback $K^{\rm nu}(V) \to K^{\rm
nu}(F)$ induces injections $K_n^{\rm nu}(V) \to K^{\rm nu}_n(F) $ for
any integers $n$, where we write $K_n^{\rm nu}(-)= \pi_n(K^{\rm nu}(-))$
if and only if $K^{\rm nu}(V^\flat) \to K^{\rm nu}(F^\flat)$ has the
same property, where $(-)^\flat$ denote the tilting functor of
perfectoid algebras.
\end{proposition}
\begin{proof}
By the assumption $K^{\rm nu}(F/V) \simeq K^{\rm nu}(F^\flat /V^\flat )$
and the isomorphism $V^\flat/ \mathfrak{m}^\flat \simeq V/\mathfrak{m}$,
one has a weak equivalence $K^{\rm nu}(F^\flat /\mathfrak{m}^\flat )
\simeq K^{\rm nu}(F/\mathfrak{m} )$. The result follows from
corollary~\ref{Gersten}. \qed
\end{proof}
\begin{remark} 
By the result \cite[Theorem 3.1]{Kelly-Morrow}, in the case the (ordinal) $K$-theories, the induced map $K_n(V^\flat) \to  K_n(F^\flat)$ is injective for any integer $n$. 
\end{remark}

%

\end{document}